\documentclass[a4paper, 12pt]{amsart}
\usepackage[english]{babel}
\usepackage{amsmath,amsfonts,amssymb,amsthm,mathtools, mathabx, amscd} 
\usepackage[format=plain]{caption}
\usepackage[utf8]{inputenc}
\usepackage{graphicx}
\usepackage{hyperref}
\usepackage{cleveref}
\usepackage{enumitem}
\usepackage{setspace}
\setlist[enumerate]{leftmargin=7mm, label=\alph*)}
\usepackage[a4paper,asymmetric]{geometry}

\usepackage{amsmath, amsthm, amssymb, amsfonts, color}
\usepackage{hyperref}
\usepackage{booktabs}
\usepackage{siunitx}
\usepackage{caption}
\usepackage[a4paper,asymmetric]{geometry}
\usepackage{hyperref}
\usepackage{float}
\usepackage{tabularx} 
\usepackage{graphicx} 
\usepackage{mathtools} 
\usepackage{mathtools}
\usepackage{cleveref}
\usepackage{enumerate}
\usepackage{verbatim}
\usepackage{tikz,tikz-cd,tikz-3dplot}
\usetikzlibrary{calc}
\usetikzlibrary{graphs,graphs.standard}
\usepackage{amssymb}
\usetikzlibrary{matrix}
\usetikzlibrary{arrows}
\usepackage[noend]{algpseudocode}
\usepackage{caption}
\usepackage[normalem]{ulem}
\usepackage{subcaption}
\usepackage{comment}
\usepackage{mathrsfs}
\usepackage{enumitem}
\usepackage{dsfont}
\usepackage{tikz-cd} 
\usepackage{changepage}
\usepackage[format=plain, font=footnotesize]{caption}
\usepackage[sort]{cite}
\usepackage{blindtext}

\newcommand{\PP}{\mathbb{P}}
\newcommand{\RR}{\mathbb{R}}
\newcommand{\KK}{\mathbb{K}}
\newcommand{\CC}{\mathbb{C}}
\newcommand{\NN}{\mathbb{N}}
\newcommand{\ZZ}{\mathbb{Z}}
\newcommand{\QQ}{\mathbb{Q}}

\newcommand{\conv}{\textnormal{conv}}

\providecommand{\keywords}[1]
{
  \small	
  \textbf{\textit{Keywords---}} #1
}
\usepackage{listings}
\definecolor{codedarkgreen}{RGB}{51, 133, 4}
\definecolor{codemaroon}{RGB}{133, 5, 63}
\definecolor{codeteal}{RGB}{0, 128, 96}
\lstdefinelanguage{Macaulay2}{
        classoffset=1,
        keywords={radical, minors, matrix, gens, Ideal, map, kernel, needsPackage, i1, i2, i3, i4},
        keywordstyle={\color{blue}},
        classoffset=2,
        morekeywords={from, to, list},
        keywordstyle={\color{codemaroon}},
        classoffset=3,
        morekeywords={QQ},
        keywordstyle={\color{codedarkgreen}},
        classoffset=4,
        morekeywords={MonomialOrder, List},
        keywordstyle={\color{codeteal}}
}
\lstdefinelanguage{Julia}{
frame=single,	 
numbers=right,
classoffset=1,
keywords={using, mixed_volume},
keywordstyle={\color{blue}}
}
\usepackage{cleveref}
\usepackage[english]{babel}
\usepackage[utf8]{inputenc}
\usepackage[colorinlistoftodos]{todonotes}    
\usepackage{setspace}
\usepackage{inconsolata} 
\usepackage{float}
\usepackage{hyperref}
\hypersetup{
     colorlinks = true,
     linkcolor = blue,
     anchorcolor = blue,
     citecolor = red,
     filecolor = blue,
     urlcolor = blue
     }
\usepackage{url}
\usepackage{cancel}
\usepackage{enumitem}
\usepackage{underscore}
\setlist[enumerate]{itemsep=1mm}
\usepackage{tikz, tikz-cd}
\usepackage{verbatim}
\usepackage[ruled,inoutnumbered]{algorithm2e}
\usepackage[sort]{cite}

\theoremstyle{plain}
\newtheorem{theorem}{Theorem}[section]

\theoremstyle{definition}
\newtheorem{definition}[theorem]{Definition} 
\newtheorem{example}[theorem]{Example}

\theoremstyle{remark}
\newtheorem*{remark}{Remark}

\newcommand{\tim}[1]{\textcolor{blue}{#1}}

\DeclareMathOperator{\argmax}{argmax}
\theoremstyle{plain}
\newtheorem{thm}{Theorem}[section]
\newtheorem{lem}[thm]{Lemma}
\newtheorem{cor}[thm]{Corollary}
\newtheorem{prop}[thm]{Proposition}
\theoremstyle{definition}
\newtheorem{rem}[thm]{Remark}

\def\rin{\rotatebox[origin=c]{90}{$\vDash$}}

\begin{document}

\title[SAGBI and Gr\"obner Bases Detection]{SAGBI and Gr\"obner Bases Detection}
\keywords{SAGBI basis, Gröbner basis, Gröbner fan, computational complexity}
\subjclass[2020]{Primary 13P10;\  Secondary 68Q15, 14M25}

\author{Viktoriia Borovik} 
\thanks{V.B., E. Sh.: supported by the DFG grant 445466444.}
\author{Timothy Duff} 
\thanks{T.D.: supported by NSF DMS-2103310}
\author{Elima Shehu}
\date{}

\maketitle
\thispagestyle{empty}

\begin{abstract}
    We introduce a detection algorithm for SAGBI basis in polynomial rings, analogous to a Gröbner basis detection algorithm previously proposed by Gritzmann and Sturmfels. 
    We also present two accompanying software packages named \texttt{SagbiGbDetection} for \texttt{Macaulay2} and \texttt{Julia}. Both packages allow the user to find one or more term orders for which a set of input polynomials form either Gr\"obner basis for the ideal they generate or a SAGBI basis for the subalgebra. 
    Additionally, we investigate the computational complexity of homogeneous SAGBI detection and apply our implementation to several novel~examples.
\end{abstract}

\section{Introduction}

Gr\"{o}bner bases, first formalized by Buchberger~\cite{BUCHBERGER2006475} in~1965, are nowadays considered fundamental to both pure and applied algebra, as they provide algorithmic solutions to many different computational problems involving multivariate polynomials: for example, solving polynomial systems of equations, ideal membership/equality, and implicitization for polynomial or rational maps.
We review basic theoretical ingredients in~\Cref{sec2}, and refer to~\cite{CoxLittleOShea} for a standard introduction.
Throughout this paper, we work in the multivariate polynomial ring $R\coloneqq \KK[x_1,\dots,x_n]$ over a field $\KK $ in commuting variables $x_1, \ldots , x_n.$
The classical algorithm for computing Gr\"{o}bner bases is Buchberger's algorithm; its input is given by a finite set of polynomials in $R,$ which generate an ideal $I\subset R$.
A Gr\"{o}bner basis for $I$ is a special generating set for $I$---typically \emph{not} the input generators---which depends, at least \emph{a priori}, on the choice of a term order $\succ $ on $R.$

Suppose an ideal $I$ is given as input by generators $f_1, \ldots , f_s \in R,$ in which case we write $I = \langle f_1, \ldots , f_s \rangle .$
Prior to computing a Gr\"obner basis for $I$, a reasonable question arises: does the set $\mathcal{F} = \{ f_1, \ldots , f_s \}$ already form a Gr\"obner basis for \emph{some} term order?
An effective algorithm solving this decision problem was first given by Gritzmann and Sturmfels~\cite{GriStu93}, and we revisit their solution in this article.
We recall that for any \emph{fixed} term order $\succ ,$ Buchberger's $S$-pair criterion allows us to decide if $\mathcal{F}$ forms a Gr\"{o}bner basis with respect to $\succ$; however, there are infinitely-many term orders on $R$ as soon as $n\ge 2.$


Just as Gr\"obner bases allow us to carry out computations with polynomial ideals, the notion of a SAGBI basis for a finitely generated subalgebra~$S\subset R$, introduced independently in \cite{robbiano2006subalgebra} and \cite{kapur}, may give us valuable information about~$S$. 
For example, we have the following solution to the algebra membership problem: given a SAGBI basis $\mathcal{F} = \{ f_1, \ldots, f_s \}$ for a subalgebra~$S \subseteq R$ with respect to the term order~$\succ$, we can determine whether the polynomial~$g \in R$ belongs to~$S$ by computing the \emph{normal form} $r$ of~$g$ with respect to $\mathcal{F}$ and $\succ .$
This means 
\begin{equation}\label{normalform}
    g = q(f_1, \dots, f_s) + r, \quad q \in \KK [y_1, \ldots , y_s],
\end{equation}
where $y_1, \ldots , y_s$ are new variables, either~$\mathrm{in}_{\succ}(r) \notin \mathrm{in}_{\succ}(S)$, or~$r = 0$, and none of the monomials present in~$r$ belongs to~$\mathrm{in}_{\succ}(S)$. 
The polynomial~$q$ and normal form~$r$, as they appear in the equation \eqref{normalform}, can be computed using an adaptation of the multivariate division algorithm known as the \emph{subduction algorithm}. Pseudocode for subduction may be found in~\cite[Algorithm 11.1]{sturmfels1996grobner} or~\cite[Algorithm 1.5]{robbiano2006subalgebra}.

SAGBI bases are of general interest in algebra, as they provide a natural generalization of classical notions from invariant theory such as the straightening algorithm.
More recently, they have been studied in connection with Newton-Okounkov bodies and toric degenerations.
As shown in~\cite{KhovHomotopy}, these connections can be applied to the development of a novel numerical homotopy continuation method for solving structured systems of polynomial equations.
The setting of this numerical method is as follows: suppose we wish to solve a system of $n$ equations in $n$ unknowns, and each equation is a general linear combination of some fixed set of polynomials $f_1,\dots,f_s \in R$. 
These polynomials give a rational parametrization of a projective variety ~$X \subseteq \PP^{s-1}.$ If~$tf_1,\dots,tf_s\in S = R[t]$ form a SAGBI basis for the subalgebra they generate for an appropriate term order on $S$, then one can construct a one-parameter flat family of varieties in which the general fiber is isomorphic to~$X$ and the special fiber is a toric variety. In this family, the subalgebra generators degenerate into a set of monomials.
The resulting system may have fewer solutions than estimates based on the Bernshtein-Kushnirenko theorem, which may then allow the original system to be solved more efficiently than by the standard polyhedral homotopy method~\cite{HUBER1998767} by tracking fewer solution~paths.

Motivated by the applications above, we consider here the natural problem of \emph{SAGBI basis detection}, and provide analogues of the results in~\cite{GriStu93}.
Specifically, we provide a new SAGBI basis detection method (Algorithm~\ref{alg:SAGBI_detection}), and prove its correctness (Theorem~\ref{cor2}.)
The result is perhaps surprising, since the theory of SAGBI bases is not completely algorithmic.
Additionally, we analyze the complexity of SAGBI basis detection in the special case of homogeneous algebras, which is applicable to the polynomial system of solving applicaiton outlined above.
Finally, we provide new implementations of both algorithms for Gr\"obner basis and SAGBI detection for both the computer algebra system \texttt{Macaulay2}~\cite{M2}
and the numerical computing language \texttt{Julia}~\cite{bezanson2017julia}.
We illustrate our implementations on a variety of examples coming from applications.

\newpage 

\subsection*{Outline.} In Section \ref{sec2}, we recall some necessary background. In Section \ref{sec3}, we give a known result of Gritzmann and Sturmfels on Gröbner basis detection and analogously present an algorithm for SAGBI detection. We then discuss the complexity of these algorithms, as well as the question of an ``optimal" term order when the generators do not form a SAGBI basis for any term order. In Section~\ref{sec4}, we describe the functionality of our packages and give basic examples of their utilization. Finally, in Section \ref{sec5}, we discuss more involved problems that can be solved with the given algorithms.

\section{Background}\label{sec2}
For brevity, let us now write~$\KK[\mathbf{x}] = \KK [x_1, \ldots , x_n]$ for a polynomial ring in~$n$ indeterminates. For any given term order~$\succ$, every non-zero polynomial~$f \in \KK[\mathbf{x}]$ has a unique leading monomial, which is denoted as~$\mathrm{in}_{\succ}(f)$ and is determined by the chosen order~$\succ$. For an ideal~$I \subseteq \KK[\mathbf{x}]$, its initial ideal is defined as 
\begin{equation}\label{eq:initial-ideal}
\mathrm{in}_{\succ}(I):= \langle \mathrm{in}_{\succ}(f)\mid f\in I\rangle.
\end{equation}
A finite subset~$\mathcal{G}\subset I$ is called a \emph{Gr\"obner basis} for~$I$ with respect to~$\succ$ if~$\mathrm{in}_{\succ}(I)$ is generated by~$\{\mathrm{in}_{\succ}(g) \mid g \in \mathcal{G}\}$. 
The existence of Gr\"{o}bner basis for any ideal $I$ and term order $\succ $ follows from Hilbert's Basis Theorem, see eg.~\cite[Ch.2, Corollary 6]{CoxLittleOShea}.

An analogous notion exists for a finitely generated subalgebra~$S\subset \KK[\mathbf{x}]$.
We define the initial~algebra of $S$ to be
\begin{equation}\label{eq:initial-algebra}
\mathrm{in}_{\succ}(S):= \KK[ \mathrm{in}_{\succ}(f)\mid f\in S].
\end{equation}
\begin{definition}
A set of polynomials~$\mathcal{F}\subset S$ is called a \emph{SAGBI basis} if the leading monomials of the elements in~$\mathcal{F}$ generate the initial
algebra~$\mathrm{in}_{\succ}(S)$.    
\end{definition}
\noindent
In contrast to the well-known situation for Gr\"{o}bner bases of ideals, there is currently no known algorithmic criterion for determining if there exists a term order for which $S$ has a finite SAGBI basis. 
It should be noted that the initial algebra of a finitely generated subalgebra may not necessarily be finitely generated itself. One well-known example is the algebra 
$\KK[x+y, xy, xy^2] \subset \KK[x, y]$, which does not admit a finite SAGBI basis for any term order.

\subsection{Preliminaries on weight orders}
We turn to the representation of term orders by weight vectors. 
For any weight vector $\omega \in \RR^n$ and $f\in \KK [\mathbf{x}],$ we write $\mathrm{in}_{\omega} (f)$ for the sum of all terms $c_{\alpha } x^{\alpha}$ appearing in $f$ such hat the dot product $\langle \alpha , \omega \rangle $ is maximized.
Analogously to~\eqref{eq:initial-ideal}, $\mathrm{in}_{\omega } (I)$ is the ideal generated by $\mathrm{in}_{\omega} (f)$ for all $f\in I.$
For any given ideal or algebra, the term order used to define the initial ideal or algebra can be represented concretely by a weight vector, as specified in the next two lemmas.
\begin{lem}\cite[Proposition 1.11]{sturmfels1996grobner}
    For any term order~$\succ$ and any ideal~$I \subseteq \KK[\mathbf{x}]$, there exists some weight~$\omega \in \RR^n_{\geq 0}$ such that~$\mathrm{in}_{\succ}(I) = \mathrm{in}_{\omega} (I)$.
\end{lem}
\begin{lem}\cite[Lemma 1.5.6]{Bruns2}
Given a finitely generated subalgebra~$S \subseteq \KK[\mathbf{x}]$, such that~$S$ admits a finite SAGBI basis with respect to $\succ$, there exists a weight~$\omega$ such that~$\mathrm{in}_{\succ} (S) = \mathrm{in}_{\omega} (S)$.
\end{lem}

Next, we introduce an equivalence relation on weight orders with respect to an ideal or a finitely generated subalgebra.
\begin{definition}
We call two weights~$\omega$ and $\omega'$ equivalent with respect to an ideal~$I$  if and only if~$\mathrm{in}_{\omega}(I) = \mathrm{in}_{\omega'}(I)$. We write $\omega\sim_{I} \omega'$.

Much in the same way, we say two weights~$\omega, \omega'$ are equivalent with respect to a subalgebra~$S$  if and only if~$\mathrm{in}_{\omega}(S) = \mathrm{in}_{\omega'}(S)$. We write $\omega \sim_{S} \omega'$.
\end{definition}
\noindent
For more details on the weight representation of term orders, see \cite[Chapter 1]{sturmfels1996grobner}.

\subsection{Criteria for Gröbner and SAGBI bases}
Buchberger's well-known $S$-pair criterion (see, e.g., \cite[Ch.2, Theorem 6]{CoxLittleOShea}) gives necessary and sufficient conditions to verify whether a given set of polynomials forms a Gröbner basis. 

We now recall the SAGBI analog.
\begin{definition}\label{toric_ideal}
    Let~$\mathcal{F} = \{ f_1,\dots,f_s\}$. The following toric ideal
   ~$$\langle P \in \KK[y_1,\dots,y_s]  \mid P(\mathrm{in}_{\succ}(f_1),\dots,\mathrm{in}_{\succ}(f_s)) = 0\rangle$$
 is called the ideal of algebraic relations between leading monomials of~$f_i$'s. Every polynomial $P(y_1,\dots,y_s)$ in this ideal produces a~polynomial~$P(\mathcal{F}):=P(f_1,\dots,f_s)\in S$ by substituting~$f_i$ for $y_i$, which we call an \emph{S-polynomial} of~$\mathcal{F}$.
\end{definition}
The following theorem gives us a criterion to check whether a given set of polynomials is a SAGBI basis for a specified term order. 
\begin{theorem}\cite[Theorem 2.8]{robbiano2006subalgebra}\label{sagbi_crit}
Given~$S = \KK[\mathcal{F}] \subset \KK[\mathbf{x}]$,~$\mathcal{F}$ is a SAGBI basis of $S$ if and only if every S-polynomial has $\mathcal{F}$-remainder $r=0$ as in eq.~\eqref{normalform}.
\end{theorem}
\noindent
In fact, it is sufficient to check the conditions of Theorem~\ref{sagbi_crit} for only those $S$-polynomials corresponding to any finite set generating the toric ideal of relations of leading terms of $\mathcal{F}.$
This can in turn be checked using the subduction algorithm.

The toric ideal of relations on leading terms is also the kernel of the map
\begin{align*}
    \KK[y_1,\dots,y_s] & \to \KK[\mathbf{x}]\\
    y_i & \mapsto \mathrm{in}_{\succ}(f_i(\mathbf{x})).
\end{align*}
Following the notation of \cite{sturmfels1996grobner}, we denote this ideal by~$I_{A}$, where~$A$ is the~$n \times s$ matrix whose columns are the vectors~$\alpha_i \in \NN^{n}$ occuring as leading exponents of the~$f_i$, i.e.,~$\mathrm{in}_{\succ}(f_i(\mathbf{x})) = \mathbf{x}^{\alpha_i}$.
Similarly, let~$I$ be the kernel of the map
\begin{equation}\label{param_map}
\begin{aligned}
    \KK[y_1,\dots,y_s] & \to \KK[\mathbf{x}]\\
    y_i & \mapsto f_i(\mathbf{x}),
\end{aligned}
\end{equation}
which defines a unirational variety parametrized by the $f_i$. Here is another strategy to verify a~SAGBI basis for a given term order.
\begin{theorem}\label{sagbi_crit2}
\cite[Theorem 11.4]{sturmfels1996grobner} Let~$\omega \in \RR^n$ be a weight vector that agrees with a term order~$\succ$ on a finite set $\mathcal{F}\subset \KK[\mathbf{x}]$. The set~$\mathcal{F}$ is a SAGBI basis of~$S = \KK[\mathcal{F}]$ with respect to $\succ$ if and only if~$\mathrm{in}_{A^T\omega}(I) = I_A$.
\end{theorem}
\begin{remark}
     Since~$A^T\omega$ is a special weight, the ideal~$\mathrm{in}_{A^T\omega}(I)$ need not be a monomial ideal and the \emph{initial form}~$\mathrm{in}_{A^T\omega}(f)$, i.e., the sum of all terms~$C\mathbf{x}^{\alpha}$, such that the inner product~$A^T\omega \cdot \alpha$ is maximal, need not be a monomial.
\end{remark}

\subsection{Newton polyhedra}
\begin{definition}
The \emph{Newton polytope~$\mathrm{New}(f)$ of~$f = \sum_{i=1}^d c_i \mathbf{x}^{\alpha_i}$} is a convex hull of the exponents of its monomials, i.e.,
$\mathrm{New}(f) = \mathrm{conv}\{ \alpha_1,\dots, \alpha_d\}$.
Its Minkowski sum with the negative orthant is called the 
 \emph{affine Newton polyhedron} of~$f$, which we denote by~$\mathrm{New_{aff}}(f):= \mathrm{New}(f) + \RR_{\leq 0}^n$.    
\end{definition}
\noindent
We now generalize this notion to the set of polynomials~$\mathcal{F} = \{ f_1,\dots,f_s\}$. The \emph{Newton polytope of~$\mathcal{F}$} is the Minkowski sum of the Newton polytopes $\mathrm{New}(f_i)$:
\[\mathrm{New}(\mathcal{F}) := \mathrm{New}(f_1) + \dots + \mathrm{New}(f_s).\]
Equivalently,~$\mathrm{New}(\mathcal{F}) = \mathrm{New}(f_1\cdots f_s)$. And similarly, we define the \emph{affine Newton polyhedron of~$\mathcal{F}$} to be the Minkowski sum
$\mathrm{New_{aff}}(\mathcal{F}):= \mathrm{New}(\mathcal{F}) + \RR_{\leq 0}^n$.

\begin{definition}
Consider a finite set of polynomials~$\mathcal{F} = \{ f_1,\dots,f_s\}$. We call two term orders~$\omega, \omega'$ equivalent with respect to~$\mathcal{F}$ if and only if~$\mathrm{in}_{\omega}(f_i) = \mathrm{in}_{\omega'}(f_i)$ for each~$i=1,\dots,s$. We write $\omega\sim_{\mathcal{F}} \omega'$.
\end{definition}

Equivalence classes with respect to a finite set of polynomials enjoy a nice relationship with the corresponding affine Newton polytope.

\begin{lem}\cite[Proposition 3.2.1]{GriStu93} \label{polyhedron_lem}
    The vertices of the
affine Newton polyhedron 
$$\mathrm{New_{aff}}(\mathcal{F}) = \mathrm{New}(\mathcal{F}) + \RR_{\leq 0}^n$$ are in one-to-one correspondence with the equivalence
classes of term orders with respect to~$\mathcal{F}$. More precisely, the open polyhedral cones of equivalent term orders are precisely the normal cones of $\mathrm{New_{aff}}(\mathcal{F}).$
\end{lem}

\section{SAGBI and Gröbner bases detection}\label{sec3}
\subsection{Gr\"obner basis detection} Let us consider the following decision problem. Given a finite set of polynomials~$\mathcal{G} \subset \KK[\mathbf{x}]$, we want to decide whether $\mathcal{G}$ is a Gröbner basis for~$I \coloneqq \langle \mathcal{G} \rangle$ with respect to any of the term orders on $\KK [\mathbf{x}].$ 
Ideally, we would like an algorithm that somehow identifies \emph{all} terms orders such that $\mathcal{G}$ is a Gr\"{o}bner basis. Such an algorithm was originally described in~\cite{GriStu93}.
We produce pseudocode for their algorithm in Algorithm~\ref{alg:gb_detection} below.

\smallskip

\vspace{0.2cm}
\begin{algorithm}
\caption{Gröbner Basis Detection}
\BlankLine
\KwIn{A finite set~$\mathcal{G} \subset \KK[\mathbf{x}]$ of polynomials.}

\KwOut{Term orders~$\omega \in \RR^n$ for which $\mathcal{G}$ forms a Gröbner basis for~$I.$}

\BlankLine
Compute the Newton polytope $\mathrm{New}(\mathcal{G})$\;
\BlankLine
Initialize an empty \emph{list}\;
\BlankLine
\For{each vertex $v$ of $\mathrm{New}(\mathcal{G})$}{
    \If{the normal cone to $v$ intersects the positive orthant $\mathbb{R}_{\geq 0}$}{
        \For{any $\omega$ in this intersection}{
        
        Check if $\mathcal{G}$ is a Gröbner basis with respect to $\omega$ using Buchberger's $S$-pair criterion\;
        
        \If{$\mathcal{G}$ is a Gröbner basis}{
            Add $\omega$ to the \emph{list};}
        }
    }
}  
\BlankLine
\Return{list}\label{alg:gb_detection}
\end{algorithm}
\vspace{0.2cm}
\begin{theorem} \cite{GriStu93}
    The following algorithm~\ref{alg:gb_detection} is correct.
\end{theorem}

\subsection{SAGBI basis detection} Similarly, the problem of SAGBI detection is another decision problem. Given a finite set $\mathcal{F} \subset \KK[\mathbf{x}]$ of polynomials, we want to decide whether there exists a term order such that $\mathcal{F}$ is a SAGBI basis of~$S:=\KK[\mathcal{F}]$. The objective is to have an algorithm capable of identifying term orders, preferably all of them, that meet the Subduction criterion.



\vspace{0.2cm}

In Algorithm \ref{alg:SAGBI_detection} presented below we outline the solution to this problem that's similar to what we've seen in Algorithm \ref{alg:gb_detection}.
\begin{example}\label{ex::1}
    Consider a subalgebra $S = \KK[\mathcal{F}]=\{ x^2 + y^2, xy, y^2\}$. The Newton polytope of $\mathcal{F}$ is a segment with endpoints $(3, 3)$ and $(1,5)$. We illustrate the affine Newton polyhedron  $\mathrm{New_{aff}}(\mathcal{F})$ with shaded normal cones to its vertices. 
    
    \bigskip
    
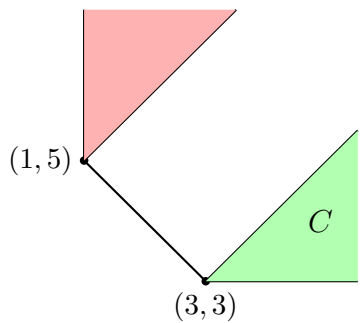
\begin{minipage}{0.4\textwidth}
\begin{center}
\begin{tikzpicture}[scale = 0.8]
  \coordinate (A) at (3,3);
  \coordinate (B) at (1,5);
  \draw[thick] (A) -- (B);
  \fill (A) circle (2pt);
  \fill (B) circle (2pt);
  \node[below] at (A) {\small{$(3,3)$}};
  \node[left] at (B) {\small{$(1,5)$}};
  \draw[thick] (B) -- ++(0,2.5) coordinate (D);
  \draw[thick] (A) -- ++(2.5,0) coordinate (P);
  \draw[thick] (B) -- ++(2.5,2.5) coordinate (C);
  \draw[thick] (A) -- ++(2.5,2.5) coordinate (T);
  \fill[red!30] (B) -- (C) -- (D) -- cycle;
  \fill[green!30] (A) -- (P) -- (T) -- cycle;
  \node[right] at (4.5,4) {$C$};
\end{tikzpicture}
\end{center}
\end{minipage}
\begin{minipage}{0.57\textwidth}
The set $\mathcal{F}$ doesn't form a SAGBI with respect to any weight $\omega$ of the red cone, but it does form a~SAGBI with respect to any weight~$\omega$ of the green cone~$C = \{\omega \mid \omega_1 > \omega_2 \geq 0\}$. The output of the Algorithm \ref{alg:SAGBI_detection} for $\mathcal{F}$ will be some~$\omega \in C$.
\end{minipage}
\end{example}

\begin{algorithm}
\caption{SAGBI Basis Detection}
\SetAlgoLined
\BlankLine
\KwIn{A finite set~$ \mathcal{F} \subset \KK[\mathbf{x}]$ of polynomials.}
\KwOut{Term orders~$\omega \in \RR^n$ for which $\mathcal{F}$ forms a SAGBI basis for~$S.$}
\BlankLine
Compute the Newton polytope~$\mathrm{New}(\mathcal{F})$\;
\BlankLine
Initialize an empty \emph{list}\;
\BlankLine
\For{each vertex~$v$ of~$\mathrm{New}(\mathcal{F})$}{
\If{the normal cone to~$v$ intersects the positive orthant~$\RR_{\geq 0}$}{
    
    \For{any $\omega$ in this intersection}{
    
    Check if~$\mathcal{F}$ is a SAGBI basis with respect to~$\omega$ using Theorems~\ref{sagbi_crit} or \ref{sagbi_crit2}\;
    
    \If{$\mathcal{F}$ is a SAGBI basis}{Add~$\omega$ to the \emph{list}\;}
    }
}
}  
\Return{list}\label{alg:SAGBI_detection}
\end{algorithm}
\begin{theorem}\label{cor2}
    The following algorithm \ref{alg:SAGBI_detection} is correct.
\end{theorem}

Following the usual definitions used for Gr\"{o}bner bases, we establish the following assertions:
\begin{definition}
   A finite SAGBI basis $H = \{ h_1, \ldots , h_k \}$ with respect to a term order $\succ $ is called \emph{reduced} if each $h_i$ is such that its leading coefficient is $1$ and none of its terms are monomials in any $\mathrm{in}_{\succ} (h_j)$.
\end{definition}
\begin{prop}\cite[Section 6.6]{RobbianoCC2}
Whenever a finite SAGBI basis exists, a~unique reduced SAGBI basis for the same term order also exists.
\end{prop}
To prove the above algorithm is correct, we will use the following results.
\begin{lem}
Consider an equivalence class of weight vectors $C[\omega ]$ with respect to a finitely generated subalgebra $S\subset \KK [\mathbf{x}]$, with the property that $S$ admits a finite SAGBI basis for some $\omega \in C[\omega ].$ The set $C[w]$ is then a relatively open convex polyhedral cone in $\mathbb{R}^n.$
\end{lem}
\begin{proof}
 Assume~$\omega \sim_{S} \omega'$, i.e., the two finitely generated monomial algebras~$\mathrm{in}_{\omega}(S)$ and~$\mathrm{in}_{\omega'}(S)$ are the same. Consider a reduced SAGBI basis~$H:=\{h_1,\ldots,h_k\}$ for~$\mathrm{in}_{\omega}(S)=\mathrm{in}_{\omega'}(S)$ w.r.t.~$\omega$. For any polynomial~$h\in H$ its leading monomial~$\mathrm{in}_{\omega'}(h)$ lies in the algebra~$\mathrm{in}_{\omega}(S)$ and thus can be represented as a monomial in~$\mathrm{in}_{\omega}(h_1),\ldots,\mathrm{in}_{\omega}(h_k)$. If~$\mathrm{in}_{\omega'}(h)$ is a monomial of~$h$ different from~$\mathrm{in}_{\omega}(h)$, the SAGBI basis~$H$ is not reduced and we have a contradiction.
Therefore,~$\mathrm{in}_{\omega}(h) = \mathrm{in}_{\omega'}(h)$ for all~$h\in H$ and each equivalence class of weight vectors is the relatively open convex polyhedral cone
\begin{equation}\label{cone}
    \begin{aligned}
        C[\omega] = & \{\omega' \in \RR^n \colon \mathrm{in}_{\omega}(S) = \mathrm{in}_{\omega'}(S) \}\\
        = & \{\omega' \in \RR^n \colon \mathrm{in}_{\omega}(h) = \mathrm{in}_{\omega'}(h) \text{ for all } h \in H\},
    \end{aligned}
\end{equation}   
cut out by strict linear inequalities in $\omega '$ enforcing that $\mathrm{in}_{\omega '}(h)$ is greater than $\mathrm{in}_{\omega '} (t)$ all other terms $t$ of $h$ for all $h\in H.$
\end{proof}
We observe that if $\mathcal{F}$ is a SAGBI basis with respect to $\omega$, we have the inclusion,
\[      \{\omega' \in \RR^n \colon \mathrm{in}_{\omega}(f) = \mathrm{in}_{\omega'}(f) \text{ for all } f\in \mathcal{F}\} \subseteq C[\omega]\]
and the preceding proof shows that equality holds when $\mathcal{F}$ is reduced.
In general, to obtain a reduced SAGBI basis we may repeatedly apply the following two operations to $\mathcal{F}$: (1) remove all polynomials $f\in \mathcal{F}$ such that either the leading term $\mathrm{in} (f)$ is products of leading terms of elements from $\mathcal{F}\setminus{\{f\}}$, and (2) update any $f$ whose tail $f - \mathrm{in} (f)$ contains a term that is a monomial in $\mathrm{in}_{\omega}(f)$ for $f \in \mathcal{F}$, we subduce this tail by $\mathcal{F}$. Note that the second operation does not change the set of leading terms, and the first operation removes one condition $\mathrm{in}_{\omega}(f) = \mathrm{in}_{\omega'}(f)$, thus weakening the defining inequalities for the cone of equivalence classes $\sim_{\mathcal{F}}$. Thus, we have the following corollary.
\begin{cor}\label{cor_cones}
    Let $\mathcal{F}$ be a finite set of polynomials and $\omega \sim_{\mathcal{F}} \omega'$.
If $\mathcal{F}$ is a SAGBI basis of $\KK[\mathcal{F}]$ with respect to $\omega$
 then $\mathcal{F}$ is a SAGBI basis of $\KK[\mathcal{F}]$
 with respect to~$\omega'$. 
\end{cor}


\begin{proof} \emph{(of Theorem \ref{cor2})}
The proof is based on Lemma \ref{polyhedron_lem} and Corollary \ref{cor_cones}.   

Algorithm \ref{alg:SAGBI_detection} iterates over all normal cones to the vertices of $\mathrm{New_{aff}}(\mathcal{F})$, checking all possible equivalence classes with respect to~$\mathcal{F}$. Corollary \ref{cor_cones} states that this will be enough to find all weight vectors $\omega$ with respect to which $\mathcal{F}$ is a SAGBI basis. Indeed, assume there is some $\omega$ such that $\mathcal{F}$ forms a SAGBI w.r.t.~$\omega$ and $\omega$ is not in the output of the Algorithm \ref{alg:SAGBI_detection}. Consider the normal cone of the affine Newton polyhedron corresponding to $\omega$,
\[ \{\omega' \in \RR^n \colon \mathrm{in}_{\omega}(f) = \mathrm{in}_{\omega'}(f) \text{ for all } f\in \mathcal{F}\}\label{cone2}. \]
By Corollary \ref{cor_cones} $\mathcal{F}$ forms a SAGBI with respect to any weight $\omega'$ from this cone and $\omega \sim_{\KK[\mathcal{F}]} \omega'$. Thus, some $\omega'$ from the cone \eqref{cone2} should be in the output of the Algorithm \ref{alg:SAGBI_detection} and we can consider this weight as a representative of the equivalence class with respect to the subalgebra $\KK[\mathcal{F}]$.
\end{proof}
\begin{rem}
As opposed to ideals, not every subalgebra~$S \subseteq \KK[\mathbf{x}]$ has finitely many distinct initial subalgebras, so there might be infinitely many cones of the form~\eqref{cone}. For such an example, we refer the reader to \cite[Example 5.3]{kuroda}.
\end{rem}
\subsection{Connection to tropical geometry}
The SAGBI detection algorithm is applied to answer the question: \emph{is there a term order $\omega$ such that the given generators of a finitely generated subalgebra of a polynomial ring are the SAGBI basis with respect to~$\omega$?} A similar question is addressed by Kaveh and Manon \cite{KM19} in the more general setting of Khovanskii bases. 
The difference between their approach and ours comes down to how the input is given.
Our input is given by the parametrization of a unirational variety, while the answer of Kaveh and Manon is given in terms of the ideal defining this variety. 

More precisely, consider a finite set of polynomials~$\mathcal{F} = \{f_1,\ldots,f_s\}$. Let $I$ be the kernel of the map~\eqref{param_map}. The \emph{tropical
variety} $\mathcal{T}(I)$ associated to $I$ is the set of vectors $v \in \RR^s$ whose associated ideal
of initial forms $\mathrm{in}_{v}(I)$ contains no monomials. $\mathcal{T}(I)$ is endowed with
the structure of a polyhedral fan since it is supported on a closed subfan of a Gröbner fan of $I$. We call $v \in \mathcal{T}(I)$ a \emph{prime point} of $\mathcal{T}(I)$ when $\mathrm{in}_{v}(I)$ is a prime binomial ideal,
and the open face $\sigma$ of the Gröbner fan of~$I$ containing $v$ in its relative interior is a \emph{prime cone}. Then by \cite{KM19}, the existence of a term order $\omega$ w.r.t.~which $\mathcal{F}$ is the SAGBI basis is equivalent to the existence of a prime cone $\sigma$ in the tropical variety $\mathcal{T}(I)$. Now let $\omega$ be from the output of Algorithm \ref{alg:SAGBI_detection} and $C$ be the corresponding intersection of the normal cone to some vertex of $\mathrm{New}(\mathcal{F})$ with $\RR_{\geq 0}$. Then applying Theorem \ref{sagbi_crit2} we obtain that $A^{T} \cdot C$ gives us a prime cone of the tropical variety $\mathcal{T}(I)$, where~$A$ is the~$n \times s$ matrix whose columns are the vectors~$\alpha_i \in \NN^{n}$ occuring as leading exponents of the~$f_i$, i.e.,~$\mathrm{in}_{\succ}(f_i(\mathbf{x})) = \mathbf{x}^{\alpha_i}$.
\begin{example}
    Consider the projective plane curve given by the solution set of the equation~$z_1z_3 - z_3^2 - z_2^2 = 0$. The tropical variety $\mathcal{T}$ of $z_1z_3 - z_3^2 - z_2^2$ is the union of the three half-planes $\QQ(1,1,1) + \QQ_{\geq 0}(0,1,0)$, $\QQ(1,1,1) + \QQ_{\geq 0}(1,0,0)$, $\QQ(1,1,1) + \QQ_{\geq 0}(-2,-1,0)$ with initial forms $z_1z_3 - z_3^2$, $z_3^2 + z_2^2$ and $z_1z_3 - z_2^2$, respectively.

    Take the parameterization of the curve given by
    $(x, y) \mapsto [x^2 + y^2: xy: y^2]$. From the Example \ref{ex::1} we have that the set $\mathcal{F} = \{ x^2 + y^2, xy, y^2\}$ forms a~SAGBI basis with respect to any weight $\omega$ from the cone $C$ is~$\{\omega \mid \omega_1 > \omega_2 \geq 0\}$. By multiplying all such vectors $\omega$ by matrix $A^{T} = \begin{footnotesize}\begin{pmatrix}
        2 & 1 & 0\\
        0 & 1 & 2 
    \end{pmatrix} \end{footnotesize}$ we obtain a half-plane $\{(2\omega_1,  \omega_1 + \omega_2, 2\omega_2) \mid \omega_1 > \omega_2 \geq 0\}$ consisting of vectors opposite to the vectors from the half-plane $\QQ(1,1,1) + \QQ_{\geq 0}(-2,-1,0)$. The vectors from this half-plane define a flat toric degeneration of our curve to the rational normal quadric curve.

    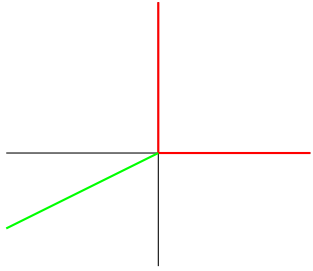
\begin{figure}
\begin{tikzpicture}
  \draw[-] (0,-1.5) -- (0,2);
  \draw[-] (-2,0) -- (2,0);
  \draw[thick, red] (0,0) -- (0,2);
  \draw[thick, red] (0,0) -- (2,0);
  \draw[thick, green] (0,0) -- (-2,-1);
\end{tikzpicture}
\caption{The tropical variety $\mathcal{T}/\QQ(1,1,1)$.}
\end{figure}
    The sign difference results from using $\max$ convention for valuation coming from term order and $\min$ convention for the tropical variety.
\end{example}
\subsection{Computational Complexity}
In this section, we analyze the complexity of SAGBI basis detection. 
Our main result is an analogue of\cite[Theorem 3.2.6]{GriStu93}, which provides upper bounds on the complexity of Gr\"{o}bner basis detection.
This previous result states that Algorithm~\ref{alg:gb_detection} runs in $\mathcal{O}(s^{n+2} m^{2n-1} d^{(2n+1)n})$ arithmetic operations, where $s$ denotes the number of input polynomials, each polynomial has at most~$m$ monomials of total degree at most~$d$, and~$n$ is the number of variables.
Thus, under a unit-cost computation model in the regime where the number of variables $n$ in the input is fixed, Gr\"{o}bner bases can be detected in polynomial time.

In~\Cref{thm:sagbi-complexity} below, we prove an analogous result for the SAGBI basis detection problem for homogeneous polynomial algebras.
Our goal is not to provide the sharpest upper bounds possible. 
Rather, we seek to show that SAGBI basis detection is fixed-parameter tractable, just as in the case of Gr\"{o}bner bases.

\begin{theorem}\label{thm:sagbi-complexity}
For fixed $n$, homogeneous SAGBI basis detection can be solved in
\[\begin{split}
   O\left( {{s + s^2d^n} \choose s} \cdot m^{2n-1} \cdot s^{4n^2+11n+4} \cdot d^{2n^3+8n^2+5n+1}\right)
\end{split}
\]
arithmetic operations.    
\end{theorem}
\begin{proof}    By \cite[Theorem 2.3.7]{GriStu93}, we can compute the normal fan of $\mathrm{New}(\mathcal{F})$ in at most~$O(s^{n}m^{2n-1})$ arithmetic operations.
Let $\omega $ be a weight vector from one of the normal cones.

According to Theorem \ref{sagbi_crit}, to verify that $\mathcal{F}$ forms a SAGBI basis with respect to~$\omega$ we should find a generating set of binomials for the ideal $I_A$, where $A$ is the~$n \times s$ matrix whose columns are the vectors $\alpha_i \in \NN^n$
occurring as leading exponents of the $f_i$, i.e., $\mathrm{in}_{\omega}(f_i(\mathbf{x})) = \mathbf{x}^{\alpha_i}$. First, we bound the total degree of generators of $I_A$. For this we consider the \emph{Castelnuovo-Mumford regularity} $reg(I_A)$, which is the maximum of the numbers $\deg(\sigma) - i$ where $\sigma$ is
a minimal $i$-th syzygy of $I_A$. 
Since the $i = 0$-th syzygies are simply the minimal generators of $I_A,$ $reg(I_A)$ is an upper bound for the maximal total degree of a~generator of $I_A$.

Let $I_A$ define a $n_A$-dimensional toric variety in $\PP^{s-1}$.
By \cite[Theorem 4.5]{Sturmfels1996EquationsDT}, $$reg(I_A) \leq s \cdot \mathrm{codim}(I_A) \cdot \deg(I_A) \le s^2 \cdot \deg(I_A).$$
The lattice $\mathbb{Z} A \subset \mathbb{Z}^n$ spans a vector space of dimension $n_A + 1$.
By Kushnirenko's Theorem, $\deg(I_A)$ equals the normalized volume of $\operatorname{conv} (A)$ inside of this lattice, see e.g. \cite[Section 3.1.2]{Sottile2006RealST}.
By assumption, $|\alpha_i| = \alpha_{i,1} +\ldots+\alpha_{i,n}$ is bounded by~$d$ for each $\alpha_i$, and hence $\conv (A)$ is contained in a lattice simplex of normalized volume~$d^{n_A} \le d^n.$
Therefore, generators of $I_A$ have degrees at most $R:=s^2 d^n$.


A crude estimate for the number of generators of $I_A$ is simply the number of monomials in variables $y_1,\ldots,y_s$ of degree at most $R$, i.e., ${{s + R} \choose s}$.

Consider now a single binomial generator of $I_A$ with degree at most $R$.
After substituting each of its variables with the corresponding $f_i$, we obtain a polynomial~$p(\mathbf{x})$ of degree at most~$dR$. 
Applying a single iteration of the subduction algorithm, we must find (or determine the non-existence of) $v_1,\ldots, v_s \in \mathbb{Z}_{\ge 0}$ s.~t.
$$\mathrm{in}_{\omega}(p(\mathbf{x})) = \mathrm{in}_{\omega}(f_1(\mathbf{x}))^{v_1}\cdots\mathrm{in}_{\omega}(f_s(\mathbf{x}))^{v_s}.$$
This is equivalent to finding a solution to the integer linear program $A\mathbf{v} = \mathbf{b}$, where all $v_i \geq 0$. Note that the columns of $A$ are the exponents in the monomial $\mathrm{in}_{\omega}(f_i(\mathbf{x}))$ and $\mathbf{b}$ is the exponent of $\mathrm{in}_{\omega}(p(\mathbf{x}))$. Thus, both $A$ and $\mathbf{b}$ have entries from $\{0, 1, \ldots, dR\}$. Then, by \cite[Corollary 1]{ComplexityIntegerProg}, there is an algorithm for solving this integer program which can be carried out in time $$O(s^{2n+2}(ndR)^{(n+1)(2n+1)}).$$ After a subduction step, the new polynomial $p'(\mathbf{x}) : = p(\mathbf{x}) - f_1(\mathbf{x})^{v_1}\cdots f_s(\mathbf{x})^{v_s}$, if nonzero, has a leading monomial $\mathbf{x}^{\mathbf{c}}$ with 
\begin{equation}\label{latticebreadth}
\langle w, \mathbf{c} \rangle \leq \langle w, \mathbf{d} \rangle,    
\end{equation}
where $\mathbf{d} \in \ZZ^n$ is the leading exponent on the previous subduction step. By \cite[Lemma 3.2.4]{GriStu93},
we can assume that $\omega$ is integral and has coordinates whose absolute values are bounded from above by $(2nd)^{2n}$. Thus, the number of subduction steps is at most the number of solutions in $\mathbf{c}\in \mathbb{Z}_{\ge 0}^n$ to the inequality~\eqref{latticebreadth}. Applying~\cite[Lemma 3.2.5]{GriStu93}, this number is bounded by $(n\cdot dR\cdot (2nd)^{2n} + 1)^n$.

Carrying out the analysis of subduction for each toric generator of $I_A$, and all possible choices of $\omega ,$ we deduce that the total number of arithmetic operations is
\[\begin{split}
    &O(s^{n}m^{2n-1}) \cdot {{s + R} \choose s} \cdot O(s^{2n+2}(ndR)^{(n+1)(2n+1)}) \cdot (2^{2n}(nd)^{2n+1}R  + 1)^n = \\
= &\ O\left( {{s + s^2d^n} \choose s} \cdot m^{2n-1}s^{3n+2}(ns^2d^{(n+1)})^{(n+1)(2n+1)} \cdot (2^{2n}(nd)^{2n+1}s^2d^n  + 1)^n\right).
\end{split}
\]
For a fixed value of $n$, this is equivalent to
\[
O\left( {{s + s^2d^n} \choose s} \cdot m^{2n-1} \cdot s^{4n^2+11n+4} \cdot d^{2n^3+8n^2+5n+1}\right).
\]
Thus, under the additional assumptions that $s$ is fixed and all field arithmetic operations have unit cost, SAGBI detection can be solved in polynomial time.
\end{proof}

\subsection{Applications}

Our results on the correctness and complexity of homogeneous SAGBI basis detection are of potential interest in applications to the theory of Newton-Okounkov bodies.
We briefly recall the highlights of this theory.

\subsubsection{Newton-Okounkov bodies}\label{sec:NObodies}In this subsection, we assume that the polynomial subalgebra $S = \KK [\mathcal{F}]$  is equipped with a positive $\mathbb{Z}$-grading: that is,
\begin{equation*}\label{eq:grading}
S =\displaystyle\bigoplus_{k\ge 0} S_i, 
\quad 
S_0 = \KK ,
\, 
S_i S_j \subseteq S_{i+j}.
\end{equation*}
A natural grading arises when the algebra generators $\mathcal{F}$ are homogeneous polynomials.
Another common setting arises as follows: given $L \subset \KK [\mathbf{x}]$, a vector space over $\KK $ spanned by finitely-many polynomials, then we may define a graded algebra via
\begin{equation}\label{eq:AL}
S_L = \displaystyle\bigoplus_{k\ge 0} t^k L^k \subset \KK [t, \mathbf{x}].
\end{equation}
Consider a semigroup $\mathcal{S}(S_L, \succ)\coloneqq \{ \text{exponent of } \mathrm{in}_{\succ}(f) \mid f \in S_L \}$.
\begin{definition}
We define the \emph{Newton-Okounkov cone} $C(S_L, \succ)$ to be the closure of the convex hull of $\mathcal{S}(S_L, \succ)$. The \emph{Newton-Okounkov body} $\Delta (S_L,\succ)$ is defined to be the intersection of the Newton-Okounkov cone $C(S_L, \succ)$ with the plane $\{1\} \times \RR^n$.
\end{definition}
When $\mathcal{S}(S_L, \succ)$ is finitely generated, i.e., $S_L$ admits a finite SAGBI basis, then the cone $C(S_L, \succ)$ is a rational polyhedral cone and $\Delta (S_L,\succ)$ is a rational polytope.

Note that we are conducting our discussion in the special case of finitely generated subalgebras of a polynomial ring and the corresponding projective unirational varieties. In general, Newton-Okounkov bodies can be defined for a positively graded domain equipped with a valuation, see eg.~\cite{KK12} or \cite{Lazarsfeld2009}. Newton-Okounkov bodies generalize Newton polytopes for toric varieties and contain significant information about the algebra or the corresponding projective variety. 
\begin{thm} \cite[Corollary 3.2]{KK12} Let $X = \mathrm{Proj}(S_L) \subset \PP^{\dim L -1}$. Then the dimension $d$ of the convex body~$\Delta (S_L,\succ)$ equals the dimension of the projective variety $X$, and the $d$-dimensional Euclidean volume of~$\Delta (S_L,\succ)$ multiplied by 
$$\frac{d!}{\mathrm{ind}(\mathcal{S}(S_L, \succ))}$$ equals the degree of $X$. 
Here, $\mathrm{ind}(\mathcal{S}(S_L, \succ))$ refers to the index of the sublattice spanned by $\mathcal{S}(S_L, \succ)$ inside~$\ZZ^n$.
\end{thm}
\subsection{Optimal term orders}
For most choices of input, Algorithms~\ref{alg:gb_detection} and~\ref{alg:SAGBI_detection} will return a negative answer. Nevertheless, among all the equivalence classes of term orders with respect to $\sim_{\mathcal{F}}$ found by these algorithms, it is still possible to choose the ``best" with respect to some criterion. Gritzmann and Sturmfels suggest measuring how far a~set of polynomials is from being a Gröbner basis by looking at the Hilbert function -- see \cite[Section 3.3]{GriStu93}. Similar ideas work for SAGBI bases in the homogeneous case. Let~$\mathcal{F}$ be a set of $s$ homogeneous polynomials on $n$ variables of total degree at most~$d$, and denote $S := \KK[\mathcal{F}]$. For two non-equivalent term orders $\omega_1$ and $\omega_2$, consider the vector
\begin{equation}\label{vectorHilb}
\left( h_{\KK[\mathrm{in}_{\omega_1}(\mathcal{F})]}(t) - h_{\KK[\mathrm{in}_{\omega_2}(\mathcal{F})]}(t)\right)_{t = 1,\ldots, s^2d^{n+1}},    
\end{equation}
where $h_{\KK[\mathrm{in}_{\omega_i}(\mathcal{F})]}(t)$ is the Hilbert function of the monomial algebra $$\KK[\mathrm{in}_{\omega_i}(\mathcal{F})] = \KK[\mathrm{in}_{\omega_i}(f_1),\ldots, \mathrm{in}_{\omega_i}(f_s)].$$
Analogously to\cite[Section 3.3]{GriStu93}, we say that $\omega_1$ is \emph{preferable} to $\omega_2$ if the first nonzero entry in the vector \eqref{vectorHilb} is positive.
We call a term order $\omega$ \emph{optimal} if it is preferable to any other term order.
\begin{lem}\label{HilbFunc}
    The set $\mathcal{F}$  is a SAGBI
basis of~$S$ with respect to a term order $\omega$ if and only if $h_{S}(t) = h_{\KK[\mathrm{in}_{\omega}(\mathcal{F})]}(t)$ for every $t = 1,\ldots, s^2d^{n+1}$.
\end{lem}
\begin{proof}
  If $\mathcal{F}$ forms a SAGBI basis, then $\KK[\mathrm{in}_{\omega}(\mathcal{F})] = \mathrm{in}_{\omega}(S)$. We have
  \[h_{\KK[\mathrm{in}_{\omega}(\mathcal{F})]}(t) = h_{\mathrm{in}_{\omega}(S)}(t) = h_{S}(t)\]
  for every $t \in \NN$, where the last equality is true due to \cite[Proposition 1.6.2]{Bruns2}.

  Conversely, suppose that $\mathcal{F}$ is not a SAGBI. Then there exists a binomial generator~$p$ of $I_A$ such that $p(f_1,\ldots,f_s)$ does not subduce to zero. This means that there exists a non-zero normal form $r$ of this polynomial as \eqref{normalform} with $\mathrm{in}_{\omega}(r) \notin \KK[\mathrm{in}_{\omega}(\mathcal{F})]$. In the proof of~\Cref{thm:sagbi-complexity}, we established that $p(f_1,\ldots,f_s)$ has degree $t \leq s^2d^{n+1}$, as does~$r$. Then, since $\left(\KK[\mathrm{in}_{\omega}(\mathcal{F})]\right)_t$ is a proper subspace of $\left( \mathrm{in}_{\omega}(S)\right)_t$, we have
  \[\dim \left(\KK[\mathrm{in}_{\omega}(\mathcal{F})]\right)_t < \dim \left( \mathrm{in}_{\omega}(S)\right)_t = \dim S_t .\]
  This contradicts the fact that $h_{S}(t) = h_{\KK[\mathrm{in}_{\omega}(\mathcal{F})]}(t)$ for every $t = 1,\ldots, s^2d^{n+1}$.
\end{proof}
The next result follows immediately.
\begin{cor}
    Let $\mathcal{F}$ be a SAGBI basis with respect to $\omega$, then $\omega$ is optimal.
\end{cor}
\begin{rem}
    Lemma \ref{HilbFunc} gives another criterion for detecting SAGBI bases with respect to a particular term order. It is often more efficient than using Theorem~\ref{sagbi_crit}. In our implementation, we have both methods for SAGBI verification.
\end{rem}
It is interesting to consider other measures of closeness to being a SAGBI basis. Fix term order $\omega$ and let $\mathbb{V}(I_{A_{\omega}})$ be a toric variety corresponding to the ideal $I_{A_{\omega}}$ as in Theorem \ref{sagbi_crit2}. Let $P_{\omega}$ be the convex hull of all lattice points from $A_{\omega}$.
\begin{definition}\label{nice_term_order}
We will call a term order $\omega_1$ \emph{nicer} than $\omega_2$ when the dimension of the projective toric variety $\mathbb{V}(I_{A_{\omega_1}})$ is greater than that of $\mathbb{V}(I_{A_{\omega_2}})$, or when the dimensions are equal and the degree of $\mathbb{V}(I_{A_{\omega_1}})$ is greater than that of~$\mathbb{V}(I_{A_{\omega_2}})$. 
\end{definition}
Niceness is easier to verify than preferability since one only needs to compute the degree and dimension of the variety $\mathbb{V}(I_{A_{\omega}})$, which in the case of a toric variety is just the normalized volume and dimension of the associated polytope $P_{\omega}$. 

Assume $\mathcal{F}$ forms a
homogeneous SAGBI basis for $S = \KK[\mathcal{F}]$ with respect to some term order~$\omega$. Then the projective varieties $\mathbb{V}(I_{A_{\omega}})$ and $\mathrm{Proj}(\mathrm{in}_{\omega}(S))$ coincide. In particular, their degrees and dimensions are equal. This means that the dimension of $\mathbb{V}(I_{A_{\omega}})$  and the volume of $P_{\omega}$ are maximal among all possible values. 
That is, the term order $\omega$ is the \emph{nicest}.
However, $\mathcal{F}$ may not form a SAGBI basis with respect to any of the nicest term orders $\omega'$. This is only true if the toric variety $\mathbb{V}(I_{A_{\omega'}})$ is normal: the maximal possible dimension and volume of $P_{\omega'}$ will establish that the polytope $P_{\omega'}$ coincides with the polytope associated to the toric variety $\mathrm{Proj}(\mathrm{in}_{\omega'}(S))$. 
In other words, their Ehrhart polynomials must be equal, which for normal toric varieties is equivalent to the equality of Hilbert functions.


\section{Functionality}\label{sec4}

We have implemented the main Algorithms~\ref{alg:gb_detection} and~\ref{alg:SAGBI_detection} in two software packages, both named \texttt{SagbiGbDetection}~\cite{SagbiGbDetectionSource},\cite{SagbiGbDetectionJulia}, for Macaulay2~\cite{M2} and Julia~\cite{bezanson2017julia}.
We use this section to illustrate the basic functionality of both packages.

\subsection{\texttt{Macaulay2} Package}
The function \texttt{weightVectorsRealizingGB} takes as input a list of polynomials~$\mathcal{G} = \{ g_1,\dots,g_s\}$ and returns a list of weight vectors (representing the equivalence classes of term orders with respect to~$\sim_{\mathcal{G}}$) with respect to which these polynomials form a Gr\"obner basis for the ideal $\langle \mathcal{G} \rangle$. 
Polyhedral computations are performed using the package \texttt{Polyhedra}~\cite{birkner2009polyhedra}, implemented in the top-level Macaulay2 language, and Buchberger's $S$-pair criterion relies on normal form computation implemented in the core of Macualay2.
\begin{example}
Consider the polynomial ring~$\mathbb Q[x,y,z,w]$ and the ideal defining the twisted cubic in~$\mathbb P^3.$
\begin{lstlisting}[language=Macaulay2]
i1 : needsPackage "SagbiGbDetection";
i2 : R = QQ[x,y,z,w]; T = minors(2, matrix{{x,y,z},{y,z,w}});
i3 : C = first entries gens T;
i4 : weightVectorsRealizingGB C 
o4 = {{6, 4, 4, 6}, {6, 6, 3, 5}, {5, 3, 6, 6}, {3, 6, 6, 3}}
o4 : List
\end{lstlisting}
\end{example}
\begin{example}
If the output of \texttt{weightVectorsRealizingGB} is an empty list, then the given polynomials do not form a Gr\"obner basis for any term order.
\begin{lstlisting}[language=Macaulay2]
i1 : needsPackage "SagbiGbDetection";
i2 : P = QQ[x,y]; G = {x^2+y^2-1, 2*x*y-1};
i3 : weightVectorsRealizingGB G 
o3 = {}
o3 : List
\end{lstlisting}
\end{example}
Similarly, the function \texttt{weightVectorsRealizingSAGBI} takes as input a list of polynomials~$\mathcal{F} = \{ f_1,\dots,f_s\}$ and returns a list of term orders (representing the equivalence classes w.r.t.~$\sim_{\mathcal{F}}$) with respect to which these polynomials form a~SAGBI basis for the subalgebra~$\KK[\mathcal{F}]$.
This relies on the implementation of subduction provided by the package \texttt{SubalgebraBases}~\cite{burr2024subalgebrabases}.

\begin{example}
Let~$\mathcal{F} = \{ x,xy-y^2,x^2y\}$. Let~$\succ_1$ be a term order such that~$y \succ_1 x$, then~$\mathcal{F}$ is SAGBI basis for~$\KK[\mathcal{F}]$ with respect to~$\succ_1$. If~$\succ_2$ is a term order such that~$x \succ_2 y$, then the~initial algebra~$\mathrm{in}_{\succ_2}(\KK[\mathcal{F}]) = \KK[x, xy, xy^2, \dots]$ is not finitely generated.
\begin{lstlisting}[language=Macaulay2]
i1 : needsPackage "SagbiGbDetection";
i2 : P = QQ[x,y]; S = {x, x*y-y^2, x^2*y}; 
i3 : weightVectorsRealizingSAGBI S
o3 = {{1, 2}}
o3 : List
\end{lstlisting}
\end{example}
\begin{example}
    If the result of the function \texttt{weightVectorsRealizingSAGBI} is an empty list, then the given polynomials are not a SAGBI basis for any term order.
\begin{lstlisting}[language=Macaulay2]
i1 : needsPackage "SagbiGbDetection";
i2 : P = QQ[x,y]; F = {x+y, x*y, x*y^2};
i3 : weightVectorsRealizingSAGBI F
o3 = {}
o3 : List
\end{lstlisting}
\end{example}

\subsection{\texttt{Julia} and basic examples}
In \texttt{Julia}, we support polynomial rings and polynomials through the package \texttt{Singular.jl}, an interface to the Singular computer algebra system~\cite{decker2019singular}, which handles Gr\"{o}bner basis computations.
Additionally, we rely on \texttt{Polymake} and its \texttt{Julia} interface~\cite{polymake,polymake_interface} for polyhedral computations, and our own top-level implementation of the subduction algorithm. 

The two main functions in our \texttt{Julia} package have interfaces similar to the corresponding \texttt{Macaulay2} functions.

\begin{example}
The following is \cite[Example 3.9]{sturmfels1996grobner}.
\begin{lstlisting}[language=Julia]
using Singular, SagbiGbDetection
Q, (x,y,z) = Singular.polynomial_ring(Singular.QQ, ["x","y","z"]);
G = [x^5 + y^3 + z^2 - 1, x^2 + y^2 + z - 1, x^6 + y^5 + z^3 - 1];
weightVectorsRealizingGB(G,Q)
\end{lstlisting}
\begin{lstlisting}[language=Macaulay2]
1-element Vector{Vector{fmpz}}:
 [[12, 15, 27]]
\end{lstlisting}
\end{example}
\begin{example}
The algebra of symmetric polynomials has a finite SAGBI basis with respect to any term order.
\begin{lstlisting}[language=Julia]
using Singular, SagbiGbDetection
S, (x,y,z) = Singular.polynomial_ring(Singular.QQ, ["x","y","z"]);
Q = [x + y + z, x*y + x*z + y*z, x*y*z];
weightVectorsRealizingSAGBI(Q,S)
\end{lstlisting}
\begin{lstlisting}[language=Macaulay2]
6-element Vector{Vector{fmpz}}:
 [[3, 2, 1], [2, 3, 1], [3, 1, 2], [1, 3, 2], [2, 1, 3], [1, 2, 3]]
\end{lstlisting}
\end{example}
There are also additional functions: \texttt{extractWeightVectors}, that takes as input a list of polynomials~$\mathcal{F}$ and returns a list of term orders representing the equivalence classes w.r.t.~$\sim_{\mathcal{F}}$, and \texttt{isUniversalGb} (\texttt{isUniversalSAGBI}) which verify if the given set of polynomials forms a universal Gröbner (SAGBI) basis.
\section{Applications}\label{sec5}
We conclude by illustrating the use of our packages in the context of several examples coming from applied algebraic geometry. The code examples for this section have been summarized and are now accessible in the repository \cite{eq:mathrepo}.

\subsection{Examples from Algebraic Statistics} 
To model conditional independence statements for random vectors drawn from a multivariate normal distribution,
\[X = (X_1, \dots, X_n)\sim N(\mu,\Sigma)\] 
let us define $R\coloneqq\mathbb{R}[\sigma_{ij} \mid 1 \leq i < j \leq n]$, and let $\Sigma = (\sigma_{i,j})$ be a $n\times n$ symmetric matrix whose entries are filled with indeterminates.
Given pairwise disjoint $B_1, B_2, B_3 \subseteq [n]$, the Gaussian conditional independence ideal~\cite{SULLIVANT20091502},
$$I_{B_1 \, \rin \,  B_2 \mid B_3} \subseteq R$$ is generated by all maximal minors of the submatrix $\Sigma_{B_1\cup B_3, B_2 \cup B_3}$ with rows indexed by $B_1\cup B_3$ and columns indexed by $B_2 \cup B_3.$
When $B_3=\emptyset, $ we have~a~Gaussian independence ideal corresponding to the independence statement $ B_1  \rin B_2.$
We associate an ideal to a finite list of (conditional) independence statements by summing the ideals associated to each individual statement, as in the next~example.

\begin{example}
Let $n=3,$ and $C=\{1\, \rin \, 3,1\, \rin \, 3\mid2\},$ then \[
\Sigma = 
\begin{bmatrix}
\sigma_{11} & \boxed{\sigma_{12}} & \boxed{\sigma_{13}} \\
\sigma_{12} & \boxed{\sigma_{22}} & \boxed{\sigma_{23}} \\
\sigma_{13} & \sigma_{23} & \sigma_{33} \\
\end{bmatrix},
\]
where $I_C=\langle\sigma_{13},\sigma_{12}\sigma_{23}-\sigma_{22}\sigma_{13}\rangle$.
Using Algorithm~\ref{alg:gb_detection}, we deduce that the given generators form a Gr\"obner basis with respect to $\omega = (2, 3, 2, 3)$.    
\end{example}

An analogous, but more involved, class of examples from algebraic statistics are the Sullivant-Talaska ideals.
Gr\"{o}bner bases for these ideals have only recently been understood---we refer to~\cite{conner2023sullivanttalaska} for the definitions and recent results.
We consider the simplest example below.


\begin{example}
Consider the cycle graph $C_4$ on vertices $1,\ldots, 4.$ Let us pick two vertices~$1$ and $3$. We obtain a $3 \times 3 $ submatrix $[1, 3] \times [3, 1]$ of~$\Sigma$ as follows:
\[
\begin{bmatrix}
\boxed{\sigma_{11}} & \sigma_{12} & \boxed{\sigma_{13}} & \boxed{\sigma_{14}}  \\
\boxed{\sigma_{12}} & \sigma_{22} & \boxed{\sigma_{23}} & \boxed{\sigma_{24}}   \\
\boxed{\sigma_{13}} & \sigma_{23} &  \boxed{\sigma_{33}} & \boxed{\sigma_{34}}   \\
\sigma_{14} & \sigma_{24} & \sigma_{34} & \sigma_{44}   \\
\end{bmatrix}.
\]
Similarly, we can generate up $4$ such submatrices corresponding to the circular intervals $[1,3]=\{1,2,3\},$ $[2,4]=\{2,3,4\},$ $[3,1]=\{3,4,1\},$ $[4,2]=\{4,1,2\},$
whose determinants generate the Sullivant-Talaska ideal $I_4.$
Utilizing the package \texttt{SagbiGbDetection}, we have confirmed that these generators form a Gr\"obner basis for $9$ classes of term orders.
We refer to \cite{eq:mathrepo} for the code supporting this observation.




     











% ... 










\end{example}

\subsection{Grassmannians and their generalizations}

A classical result from invariant theory~\cite[Theorem 3.2.9]{invarianttheory} implies that the maximal minors of the general matrix of indeterminates~$(x_{ij})$ of size~$k\times n$,~$k<n$, form a SAGBI basis with respect to any diagonal term order, where a term order on~$\KK[\mathbf{x}]$ is called \emph{diagonal} if the product of terms on the main diagonal is the leading term of each~$(k\times k)$-minor. The subalgebra generated by all such minors is a homogeneous coordinate algebra of the Grassmannian $\mathrm{Gr}(k,n)$ of~$k$-dimensional subspaces in an~$n$-dimensional vector space in its \emph{Plücker embedding}. For more background, see \cite[Chapter 5]{jaBernd}.

Moreover, the maximal~$(k\times k)$-minors form a \emph{universal} Gr\"obner basis for the ideal they generate, i.e., it is a  Gr\"obner basis with respect to all term orders, see \cite{conca}. On the other hand, the Pl\"ucker coordinates are a universal SAGBI basis if~$k=2$, but this property fails already for~$k=3$, see \cite[Corollary 5.6]{SpeyerSturmfels}. More experiments with algebras of maximal minors are presented in the paper \cite{Bruns}.
\begin{example} For the Grassmannian~$\mathrm{Gr}(2, 4)$
the outputs of the functions \texttt{weightVectorsRealizingSAGBI} and \texttt{weightVectorsRealizingGB} are identical and consist of a 24-element vector of weights that represents all possible term orders up to equivalence, see \cite{eq:mathrepo}.
\end{example} 
In \cite[Theorem 3.2.9]{invarianttheory}, the condition that the minors are maximal is essential. For example, in \cite[Example 11.7]{sturmfels1996grobner}, Sturmfels shows that the~$(2\times 2)$-minors of a general~$(3\times 3)$-matrix do not form a SAGBI basis with respect to the diagonal term order. The question arises: do these minors form a SAGBI basis for some other term order? 
Applying our functions, we can answer this question affirmatively.
\begin{example}
Consider a~$(3\times 3)$-matrix~$(t_{ij})$ and its~$(2\times 2)$-minors. The outputs of the functions \texttt{weightVectorsRealizingSAGBI} and  \texttt{weightVectorsRealizingGB} show that the given set of polynomials forms a SAGBI basis for $6$ classes of term orders and a Gröbner basis for 96 classes of term orders,  see \cite{eq:mathrepo}.  



\end{example}
Another generalization of Grassmannians is presented in a recent work of Faulstich, Sturmfels, Sverrisdóttir \cite{svala}. The authors explore methods of \emph{coupled cluster theory} \cite{OsterFaul} arising in quantum chemistry and introduce a new type of projective unirational algebraic varieties~$V_{\sigma}$ in~$\PP^{{n \choose k} -1}$, called \emph{truncation varieties}, one for each subset~$\sigma$ of~$[k] = \{ 1, 2, \dots, k\}$. In particular, \cite[Theorem 3.5]{svala} states that~$V_{\{1\}} = \mathrm{Gr}(k, n)$.

    Many nice properties of Grassmannian varieties follow from the result that the Pl\"ucker coordinates form a SAGBI basis. In particular, SAGBI bases give rise to toric degenerations whose centra fibers are projective toric varieties, see e.g.~\cite[Section 2.3]{HUBER1998767} or \cite{CLARKE2020646, Makhlin} for details.
We checked whether the same property holds for the truncation variety~$V_{\{ 1, 3\}}$.
\begin{example}
 Consider the truncation variety~$V_{\{ 1, 3\}}$ with the parametrization given by the set of polynomials
 \[ 
 \begin{aligned}
Q = \{ &1, z_1, z_2, z_3, z_4, z_5, z_6, z_7, z_8, z_9, \\
      &(z_1z_5 - z_2z_4), (z_1z_6 - z_3z_4), (z_2z_6 - z_3z_5),\\
      &(z_1z_8 - z_2z_7), (z_1z_9 - z_3z_7), (z_2z_9 - z_3z_8),\\
      &(z_4z_8 - z_5z_7), (z_4z_9 - z_6z_7), (z_5z_9 - z_6z_8),\\
      &z_{10} + z_1(z_5z_9 - z_6z_8) - z_2(z_4z_9 - z_6z_7) + z_3(z_4z_8 - z_5z_7)\}.
 \end{aligned} 
 \]
See \cite[Chapter 3]{svala} for the construction of truncation varieties and \cite[Example 2.3]{svala} for the parameterization. The output of the function \texttt{weightVectorsRealizingSAGBI} for $t \cdot Q$  is empty, so~$t \cdot Q$ is not a SAGBI basis for any term order; see \cite{eq:mathrepo}.
\end{example}
\begin{remark}
This example suggests a natural question: does there exist a non-trivial truncation variety, i.e.~not a Grassmannian a linear space, which admits a SAGBI basis in its exponential parametrization presented in \cite{svala}.   
\end{remark}
\subsection{Algebras Generated by Principal minors}
In this subsection, we illustrate how we can sometimes efficiently find a SAGBI basis even if the original generating set of an algebra does not form one, using our notion of a nice term order in Definition \ref{nice_term_order}. 
The example concerns the principal minors of an arbitrary symmetric matrix $A \in \CC^{3\times 3}$, which we may think of as a point in $\mathbb{P}^5.$ 
Consider the projective principal minor map
\begin{align*}
    \gamma \colon \PP^5 &\to \PP^7\\
    A = [a_{ij}]\,&\mapsto \, \hbox{all the principal minors of } A.
\end{align*}
The closure of the image $\gamma$ is a hypersurface in $\PP^7$.
As observed by Holtz and Sturmfels~\cite{holtz2007hyperdeterminantal}, its implicit equation is given by Cayley's $2\times 2 \times 2$ hyperdeterminant. 
This hypersurface also has an interpretation in computer vision (see~\cite[\S 2.1]{larsson2020calibration}) as the locus of flatlander trifocal tensors for three pinhole projections $\PP^2 \dashrightarrow \PP^1$ with collinear centers. 
It has a dimension $6$ and a degree $4.$ We denote by $\mathcal{F}$ the set of 8~principal minors multiplied by the scaling variable $t$, and by $S := \CC[\mathcal{F}] \subset \CC[t, \mathbf{a}]$ its homogeneous coordinate ring.

Using the function \texttt{weightVectorsRealizingSAGBI}, we check that $\mathcal{F}$ does not form a~SAGBI basis for any term order. The function \texttt{extractWeightVectors} computes 14 equivalence classes of term orders with respect to $\mathcal{F}$. Among them, there are only 5 equivalence classes with respect to $S$. For each $\omega$ of these 5~term orders, we compute the dimension and the degree of the toric variety $Y_{\omega} := \mathrm{Spec}(\CC[\mathrm{in}_{\omega}(\mathcal{F})])$ and the SAGBI basis for $S$ with respect to~$\omega$.
\bigskip
\begin{center}
\begin{tabular}{ccccc} \toprule
    {$\dim Y_{\omega}$} & {$\deg Y_{\omega}$} & {SAGBI in degree $\leq 5,$} & {$\leq 6$} & {cardinality of reduced SAGBI basis} \\ 
    \midrule
    6  & 3 & \texttt{false} & \texttt{false} & --- \\
    6  & 2  & \texttt{true} & \texttt{true}  & 9  \\
    5  & 3  & \texttt{false} & \texttt{true}  & 11 \\
    4  & 4  & \texttt{false} & \texttt{true}  & 14    \\  
    3  & 6  & \texttt{false} & \texttt{false}  & 25    \\
    \bottomrule
\end{tabular}
\end{center}
\smallskip
We see that, in the cases where have verified that the algebra $\CC[\mathrm{in}_{\omega}(\mathcal{F})]$ is finitely generated in degree $\le 6$, we obtain the degree-wise and cardinality-wise smallest SAGBI basis for the second weight, for which the dimension of $\mathrm{dim} Y_{\omega}$ is maximized.
Note, however, that greed is not always good; the first row corresponds to the nicest class of term orders, but in this case, there is no SAGBI basis in degree $\le 6.$

\bibliographystyle{plain}
\bibliography{lib.bib}

\begin{thebibliography}{10}

\bibitem{bezanson2017julia}
Jeff Bezanson, Alan Edelman, Stefan Karpinski, and Viral~B Shah.
\newblock Julia: A fresh approach to numerical computing.
\newblock {\em SIAM review}, 59(1):65--98, 2017.

\bibitem{birkner2009polyhedra}
Ren{\'e} Birkner.
\newblock Polyhedra: a package for computations with convex polyhedral objects.
\newblock {\em Journal of Software for Algebra and Geometry}, 1(1):11--15,
  2009.

\bibitem{eq:mathrepo}
Viktoriia Borovik, Timothy Duff, and Elima Shehu.
\newblock Example codes of this manuscript.
\newblock Available at
  \url{https://github.com/elimashehu/SagbiGbDetection_Applications.git}.

\bibitem{SagbiGbDetectionJulia}
Viktoriia Borovik, Timothy Duff, and Elima Shehu.
\newblock {SagbiGbDetection: }.
\newblock A \emph{Julia} package available at
  \url{https://github.com/V-Borovik/SagbiGbDetection.jl.git}.

\bibitem{SagbiGbDetectionSource}
Viktoriia Borovik, Timothy Duff, and Elima Shehu.
\newblock {SagbiGbDetection: A \emph{Macaulay2} package. Version~0.1}.
\newblock Available at
  \url{https://github.com/Macaulay2/M2/tree/master/M2/Macaulay2/packages}.

\bibitem{Bruns}
Winfried Bruns and Aldo Conca.
\newblock Sagbi combinatorics of maximal minors and a sagbi algorithm.
\newblock {\em {\tt arXiv.2302.14345}}, 2023.

\bibitem{Bruns2}
Winfried Bruns, Aldo Conca, Claudiu Raicu, and Matteo Varbaro.
\newblock {\em Determinants, Gr\"obner Bases and Cohomology}, volume XIII,
  Springer Monographs in Mathematics.
\newblock Springer, 2022.

\bibitem{BUCHBERGER2006475}
Bruno Buchberger.
\newblock Bruno buchberger’s phd thesis 1965: An algorithm for finding the
  basis elements of the residue class ring of a zero dimensional polynomial
  ideal.
\newblock {\em Journal of Symbolic Computation}, 41(3):475--511, 2006.

\bibitem{burr2024subalgebrabases}
Michael Burr, Oliver Clarke, Timothy Duff, Jackson Leaman, Nathan Nichols, and
  Elise Walker.
\newblock Subalgebrabases in macaulay2.
\newblock {\em To appear in Journal of Software for Algebra and Geometry},
  2024.

\bibitem{KhovHomotopy}
Michael Burr, Frank Sottile, and Elise Walker.
\newblock Numerical homotopies from khovanskii bases.
\newblock {\em Mathematics of Computation}, 92, 2020.

\bibitem{CLARKE2020646}
Oliver Clarke and Fatemeh Mohammadi.
\newblock Toric degenerations of grassmannians and schubert varieties from
  matching field tableaux.
\newblock {\em Journal of Algebra}, 559:646--678, 2020.

\bibitem{conca}
Aldo Conca, Emanuela De~Negri, and Elisa Gorla.
\newblock Universal grobner bases for maximal minors.
\newblock {\em International Mathematics Research Notices}, 2015, 2013.

\bibitem{conner2023sullivanttalaska}
Austin Conner, Kangjin Han, and Mateusz Michałek.
\newblock Sullivant-{T}alaska ideal of the cyclic gaussian graphical model.
\newblock {\em {\tt arXiv.2308.05561}}, 2023.

\bibitem{CoxLittleOShea}
David Cox, John Little, and Donal O’Shea.
\newblock {\em Ideals, Varieties, and Algorithms. An Introduction to
  Computational Algebraic Geometry and Commutative Algebra}.
\newblock 2007.

\bibitem{decker2019singular}
Wolfram' Decker, Gert-Martin' Greuel, Gerhard' Pfister, and Hans'
  Sch{\"o}nemann.
\newblock Singular 4-1-2—a computer algebra system for polynomial
  computations.
\newblock {\em http://www. singular. uni-kl. de}, 2019.

\bibitem{svala}
Fabian Faulstich, Bernd Sturmfels, and Svala Sverrisdóttir.
\newblock Algebraic varieties in quantum chemistry.
\newblock {\em {\tt arXiv.2308.05258}}, 2023.

\bibitem{OsterFaul}
Fabian~M. Faulstich and Mathias Oster.
\newblock Coupled cluster theory: Toward an algebraic geometry formulation.
\newblock {\em SIAM Journal on Applied Algebra and Geometry}, 8(1):138--188,
  2024.

\bibitem{polymake}
Ewgenij Gawrilow and Michael Joswig.
\newblock Polymake: a framework for analyzing convex polytopes.
\newblock In {\em Polytopes—combinatorics and computation}, pages 43--73.
  Springer, 2000.

\bibitem{M2}
Daniel~R. Grayson and Michael~E. Stillman.
\newblock Macaulay2, a software system for research in algebraic geometry.
\newblock Available at \url{https://macaulay2.com/}, 2020.

\bibitem{GriStu93}
Peter Gritzmann and Bernd Sturmfels.
\newblock Minkowski addition of polytopes: computational complexity and
  applications to {G}r\"{o}bner bases.
\newblock {\em SIAM J. Discrete Math.}, 6(2):246--269, 1993.

\bibitem{holtz2007hyperdeterminantal}
Olga Holtz and Bernd Sturmfels.
\newblock Hyperdeterminantal relations among symmetric principal minors.
\newblock {\em Journal of Algebra}, 316(2):634--648, 2007.

\bibitem{HUBER1998767}
B.~Huber, F.~Sottile, and B.~Sturmfels.
\newblock Numerical schubert calculus.
\newblock {\em Journal of Symbolic Computation}, 26(6):767--788, 1998.

\bibitem{polymake_interface}
Marek Kaluba, Benjamin Lorenz, and Sascha Timme.
\newblock Polymake. jl: A new interface to polymake.
\newblock In {\em International Congress on Mathematical Software}, pages
  377--385. Springer, 2020.

\bibitem{kapur}
Deepak Kapur and Klaus Madlener.
\newblock A completion procedure for computing a canonical basis for a
  k-subalgebra.
\newblock In Erich Kaltofen and Stephen~M. Watt, editors, {\em Computers and
  Mathematics}, pages 1--11, New York, NY, 1989. Springer US.

\bibitem{KM19}
K.~Kaveh and C.~Manon.
\newblock Khovanskii bases, higher rank valuations, and tropical geometry.
\newblock {\em SIAM Journal on Applied Algebra and Geometry}, 3.2:292–336,
  2019.

\bibitem{KK12}
Kiumars Kaveh and A.~G. Khovanskii.
\newblock Newton-okounkov bodies, semigroups of integral points, graded
  algebras and intersection theory.
\newblock {\em Annals of Mathematics}, 176(2):925--978, 2012.

\bibitem{RobbianoCC2}
Martin Kreuzer and Lorenzo Robbiano.
\newblock {\em Computational Commutative Algebra 2}.
\newblock 01 2005.

\bibitem{kuroda}
Shigeru Kuroda.
\newblock A new class of finitely generated polynomial subalgebras without
  finite sagbi bases.
\newblock {\em arXiv preprint arXiv:2110.08748}, 2021.

\bibitem{larsson2020calibration}
Viktor Larsson, Nicolas Zobernig, Kasim Taskin, and Marc Pollefeys.
\newblock Calibration-free structure-from-motion with calibrated radial
  trifocal tensors.
\newblock In {\em Computer Vision--ECCV 2020: 16th European Conference,
  Glasgow, UK, August 23--28, 2020, Proceedings, Part V 16}, pages 382--399.
  Springer, 2020.

\bibitem{Lazarsfeld2009}
Robert Lazarsfeld and Mircea Mustață.
\newblock Convex bodies associated to linear series.
\newblock {\em Annales scientifiques de l'École Normale Supérieure},
  42(5):783--835, 2009.

\bibitem{Makhlin}
Igor Makhlin.
\newblock Chain-order polytopes: toric degenerations, young tableaux and
  monomial bases.
\newblock 2022.

\bibitem{jaBernd}
M.~Micha{\l}ek and B.~Sturmfels.
\newblock {\em Invitation to nonlinear algebra}, volume 211, Graduate Studies
  in Mathematics.
\newblock American Mathematical Society, 2021.

\bibitem{ComplexityIntegerProg}
Christos~H. Papadimitriou.
\newblock On the complexity of integer programming.
\newblock {\em J. ACM}, 28(4):765–768, 1981.

\bibitem{robbiano2006subalgebra}
Lorenzo Robbiano and Moss Sweedler.
\newblock Subalgebra bases.
\newblock In {\em Commutative Algebra: Proceedings of a Workshop held in
  Salvador, Brazil, Aug. 8--17, 1988}, pages 61--87. Springer, 2006.

\bibitem{Sottile2006RealST}
Frank Sottile.
\newblock Real solutions to equations from geometry.
\newblock In {\em University Lecture Series}, 2006.

\bibitem{SpeyerSturmfels}
David~E. Speyer and Bernd Sturmfels.
\newblock The tropical grassmannian.
\newblock {\em Advances in Geometry}, 4:389--411, 2003.

\bibitem{invarianttheory}
Bernd Sturmfels.
\newblock {\em Algorithms in Invariant Theory}.
\newblock Springer-Verlag, Berlin, Heidelberg, 1993.

\bibitem{Sturmfels1996EquationsDT}
Bernd Sturmfels.
\newblock Equations defining toric varieties.
\newblock {\em arXiv: Algebraic Geometry}, 1996.

\bibitem{sturmfels1996grobner}
Bernd Sturmfels.
\newblock {\em Grobner bases and convex polytopes}, volume~8.
\newblock American Mathematical Soc., 1996.

\bibitem{SULLIVANT20091502}
Seth Sullivant.
\newblock Gaussian conditional independence relations have no finite complete
  characterization.
\newblock {\em Journal of Pure and Applied Algebra}, 213(8):1502--1506, 2009.
\newblock Theoretical Effectivity and Practical Effectivity of Gröbner Bases.

\end{thebibliography}

\bigskip 

\noindent
\footnotesize
{\bf Authors' addresses:}

\smallskip

\noindent Viktoriia Borovik,
Universit\"at Osnabr\"uck
\hfill {\tt vborovik@uni-osnabrueck.de}

\noindent Timothy Duff,
University of Washington
\hfill {\tt timduff@uw.edu}

\noindent Elima Shehu,
Otto-von-Guericke-Universität Magdeburg
\hfill {\tt elima.shehu@ovgu.de}
\end{document}